\documentclass[12pt, reqno, twoside, letterpaper]{amsart}

\usepackage{url}
\usepackage[colorlinks,linkcolor=blue,anchorcolor=blue,citecolor=blue]{hyperref}
\usepackage{color}
\usepackage{float}

\hypersetup{breaklinks=true}
\usepackage{paperstyle}

\def\TS{T}

\def\tV{\widetilde{V}}

\newcommand{\fb}{\ensuremath{\mathfrak{b}}}
\newcommand{\fq}{\ensuremath{\mathfrak{q}_\sharp}}

\title[Character sums and $L$-functions modulo prime powers]
      {Bounds on short character sums and $L$-functions\break
       for characters with a smooth modulus}
      
\author[W.\ D.\ Banks]{William D.\ Banks}

\address{Department of Mathematics, 
         University of Missouri, 
         Columbia MO, USA.}

\email{bankswd@missouri.edu}
      
\author[I.\ E.\ Shparlinski]{Igor E.\ Shparlinski}

\address{Department of Pure Mathematics,
		 University of New South Wales,
		 Sydney, NSW 2052, Australia.}
\email{igor.shparlinski@unsw.edu.au}
  
\date{\today}

\begin{document}

\begin{abstract}
We combine a classical idea
of Postnikov (1956) with the method of Korobov (1974) 
for estimating double Weyl sums, deriving new bounds on
short character sums when the modulus $q$ has a small core
$\prod_{p\mid q}p$. Using this estimate,  we improve
certain bounds of Gallagher (1972) and Iwaniec (1974) for
the corresponding $L$-functions. In turn, this allows us to
improve the error term in the asymptotic formula for primes 
in short arithmetic progressions modulo a  power of a fixed prime.
As yet another application of our bounds, we substantially extend the region 
free of \emph{Siegel zeros}.
\end{abstract}

\maketitle


\begin{quote}
\end{quote}

\begin{quote}
\end{quote}

\section{Introduction}

\subsection{Background}
The core (or kernel) of a positive integer $q$ is the product $\fq$
over the prime divisors $p$ of $q$, that is,
$$
\fq=\prod_{p\mid q}p.
$$
Given a modulus $q$ with a small core $\fq$, 
a nonprincipal character $\chi$ modulo $q$,
and integers $M$ and $N\ge 1$, we study
the character sum $S_\chi(M,N)$ defined by
$$
S_\chi(M,N) = \sum_{n=M+1}^{M+N}\chi(n).
$$
In the case of a prime power modulus $q=p^\gamma$,
where $\fq=p$ is prime and $\gamma$ is a large integer,
it has been known since the work of  Postnikov~\cite{Post1,Post2}
that these sums satisfy bounds that are superior to those
which can be established for arbitrary moduli
(in full generality, the Burgess bound
still gives the strongest known results;
see, e.g.,  Iwaniec and Kowalski~\cite[Theorem~12.6]{IwKow}).
Further advances and modifications have been achieved by 
Gallagher~\cite{Gal} along with
applications to $L$-functions and to the distribution of primes in 
progressions modulo $p^\gamma$.
Iwaniec~\cite{Iwan} has extended those results to moduli $q$ with
a small core $\fq$. 
Both Gallagher~\cite{Gal} and Iwaniec~\cite{Iwan} 
also give estimates for Dirichlet $L$-functions  $L(s,\chi)$
(where $s=\sigma+it\in\CC$ with $\sigma=\Re s$ and $t=\Im s$) when
$\sigma$ is close to one and $\chi$
is a primitive character modulo $q$;
their estimates are uniform in the parameters $q$ and $t$,
where $q = p^\gamma$ (in~\cite{Gal}) 
or $q$ has a small core (in~\cite{Iwan}).
Further results in this direction have
been obtained by Chang~\cite{Chang}. 

\subsection{Outline of results}

Here we combine the method of Postnikov\cite{Post1,Post2}
with a different approach to estimating exponential sums with polynomials
which is due to Korobov~\cite{Kor}.  This allows us to improve
known bounds on character sums and Dirichlet polynomials,
which in turn leads to new bounds on Dirichlet 
$L$-functions and their
zero-free regions. In particular, we improve some of the 
main results of Gallagher~\cite{Gal} and Iwaniec~\cite{Iwan}
and  substantially extend the region 
free of \emph{Siegel zeros}.
See Sections~\ref{sec:BoundChar} and~\ref{sec:Appl} below
for a precise description of our results and techniques.

Furthermore, as an application of our results on Dirichlet  $L$-functions, in Section~\ref{sec:psi int} 
we give a new asymptotic formula
for the number of primes in arithmetic progressions relative
to a large prime power modulus, including
the case in which primes are taken from a short interval.
We do not improve the Linnik exponent on the least
prime in an arithmetic progression of this type 
(see~\cite{Chang,Iwan} for the latest results in this direction)
since we are unable to exploit
the specific form of the bound~\eqref{eq:varthetashape} below
to strengthen existing zero density estimates.
Nevertheless, it is likely that Theorem~\ref{thm:zerofreeregion} will find many other interesting applications. 

\section{Bounds of character sums}
\label{sec:BoundChar}

\subsection{New bounds on short character sums}
For a given prime $p$, let $v_p$ be the standard
$p$-adic valuation; in other words, if $n\ne 0$ and $v_p(n)=\nu$,
then $\nu$ is the largest integer for which $p^\nu\mid n$.
In this paper, we show that there are absolute, effectively computable constants
$\gamma_0,\xi_0>0$ with the following property. 
For any modulus $q$ satisfying
\begin{equation}
\label{eq:minmax}
\min\limits_{p\mid q}\{v_p(q)\}
\ge 0.7\gamma\qquad
\text{with}\quad\gamma=\max\limits_{p\mid q}\{v_p(q)\}\ge\gamma_0,
\end{equation}
the bound
\begin{equation}
\label{eq:mainbound}
S_\chi(M,N) \le A  N^{1-\xi_0/\varrho^2}
\qquad(M,N\in\ZZ,~N\ge \fq^{\gamma_0})
\end{equation}
holds, where $\varrho$ is determined via the
relation $N^\varrho=q$, and  $A$ is an absolute and effective constant.

In earlier versions of this
result, the bounds have been of the somewhat weaker form
\begin{equation}
\label{eq:Iwaniec}
S_\chi(M,N)\le \exp\(a \varrho  (1+ \log\varrho)^2\) N^{1-\xi_0/(\varrho^2\log\varrho)}
\end{equation}
with an absolute constant $a$ (see, e.g.,~\cite[Theorem~12.16]{IwKow}).
One advantage of~\eqref{eq:mainbound} over~\eqref{eq:Iwaniec} is the absence of
$\log\varrho$ in the denominator of the ``savings'' term
in the exponent of $N$.  A more crucial advantage, however, is that our  
bound~\eqref{eq:mainbound} has an {\it absolute constant\/} $A$ instead 
of the superexponential function  of $\varrho$ that appears in~\eqref{eq:Iwaniec};
this ultimately accounts for our improvement of the exponent $3/4$
in~\eqref{eq:Iw cutoff} down to  $2/3$ in~\eqref{eq:BS cutoff} below.  

We note that the recent work of Chang~\cite{Chang} extends the class of moduli $q$ 
to which the method of Postnikov~\cite{Post1,Post2} applies, but provides weaker 
bounds than ours.

Mili{\'c}evi{\'c}~\cite{Mil} also uses the method 
of Postnikov~\cite{Post1,Post2}. However, the main goal of~\cite{Mil} is to
estimate $L$-functions  $L(s,\chi)$ 
in the different extreme case in which $s=1/2$, as opposed to the case $s=1$
(or more generally, $\sigma$ close to one) which is the case considered here.
It turns out that for applications to $L(1/2,\chi)$ the strength of the bound
of the character sums is more important than its range. Thus, Mili{\'c}evi{\'c}~\cite{Mil} works in a different regime of long character sums,
whereas we are mainly interested in short sums that are decisive
for estimating $L(s,\chi)$ when $\sigma$ is close to one.

To give a brief comparison of the strengths  of our  bound~\eqref{eq:mainbound},
which stems from our approach via double sums, and
of~\eqref{eq:Iwaniec}, which is based on standard Weyl sums,
we note that~\eqref{eq:mainbound} is nontrivial for
\begin{equation}
\label{eq:BS cutoff}
N \ge \exp\((\log q)^{2/3+\eps}\)
\end{equation}
whereas~\eqref{eq:Iwaniec} requires that
\begin{equation}
\label{eq:Iw cutoff}
N \ge \exp\(( \log q)^{3/4+\eps}\).
\end{equation}

Our approach to~\eqref{eq:mainbound} relies on an idea of
Korobov~\cite{Kor} coupled with the use of Vinogradov's mean value theorem in
the explicit form given by Ford~\cite{Ford}.  Specifically,
we employ a precise bound on the
quantity $N_{k,d}(P)$ defined as the number of solutions to the system of equations
\begin{equation}
\label{eq:System}
y_1^r+\cdots y_k^r=z_1^r+\cdots+z_k^r\qquad
(1\le r\le d,~1\le y_r,z_r\le P).
\end{equation}
It is worth remarking that later improvements of Vinogradov's mean value theorem
due to Wooley~\cite{Wool1,Wool2,Wool3}, and more recently, to
Bourgain, Demeter and Guth~\cite{BDG} (the latter providing a bound that is
essentially optimal with respect to  $P$), are not suitable
for our purposes as they contain implicit constants
that depend on $k$ and $d$, whereas our methods require that $k$ and $d$
be permitted to grow with $P$.

Bearing in mind potential applications to $L$-functions
(some of which are given below) we 
establish the following generalization
of the bound~\eqref{eq:mainbound}. 
For a given polynomial $G(x)$ with real coefficients, let
$$
S_\chi(M,N;G) = \sum_{n=M+1}^{M+N}\chi(n)\e(G(n)),
$$
where $\e(t)=e^{2\pi it}$ for all $t\in\RR$.

For given 
functions $U$ and $V$, the notations $U\ll V$, $V\gg U$ and
$U=O(V)$ are all equivalent to the statement that the inequality
$|U|\le c|V|$ holds with some constant $c>0$.
Throughout the paper, we indicate explicitly
the parameters on which the implied constants may depend.

\begin{theorem} 
\label{thm:mainG}
For any real number $C>0$ there are effectively computable constants
$\gamma_0,\xi_0>0$ that depend only on $C$ and have the following property.
For any modulus $q$ satisfying~\eqref{eq:minmax}
and any primitive character $\chi$ modulo $q$, the bound
\begin{equation}
\label{eq:mainboundG}
S_\chi(M,N;G) 
\ll N^{1-\xi_0/\varrho^2}
\end{equation}
holds uniformly for all $M,N\in\ZZ$ and $G\in\RR[x]$ subject to the conditions
\begin{equation}
\label{eq:condsC}
q \ge N\ge \fq^{\gamma_0}\mand \deg G\le C\varrho,
\end{equation}
where $\varrho=(\log q)/\log N$
and  implied constant in~\eqref{eq:mainboundG}
is effective and depends only on $C$.
\end{theorem}

As an application of Theorem~\ref{thm:mainG}, we also study Dirichlet polynomials
of the form
$$
\TS_\chi(M,N;t) = \sum_{n=M+1}^{M+N}\chi(n) n^{it}\qquad(t\in\RR).
$$
Approximating $\TS_\chi(M,N;t)$ by sums $S_\chi(M,N;G)$ with appropriately
chosen polynomials $G$, we derive the following bound.

\begin{theorem}
\label{thm:mainT}
For any real number $C>0$ there are effectively computable constants
$\gamma_0,\xi_0>0$ that depend only
on $C$ and have the following property.
For any modulus $q$ satisfying~\eqref{eq:minmax}
and any primitive character $\chi$ modulo $q$, the bound
\begin{equation}
\label{eq:mainboundT}
\TS_\chi(M,N;t) 
\ll N^{1-\xi_0/\varrho^2}
\end{equation}
holds uniformly for all $M,N\in\ZZ$ and $t\in\RR$ subject to the conditions
\begin{equation}
\label{eq:condsT}
2N\ge M\ge N,\qquad
q \ge N\ge \fq^{\gamma_0}\mand|t|\le q^C,
\end{equation}
where $\varrho=(\log q)/\log N$
and  implied constant in~\eqref{eq:mainboundT}
is effective and depends only on $C$.
\end{theorem}

Theorem~\ref{thm:mainT} improves~\cite[Lemma~5]{Gal} in the special case that
$|t|$ is bounded by a fixed power of the modulus of the
character $\chi$.  For larger values of $|t|$, our approach incorporating
ideas of Korobov (Lemma~\ref{lem:korobov}) breaks down; in this case, the method
of Gallagher (which relies only on general estimates of Vinogradov~\cite{Vino1,Vino2})
yields the best known result.

\section{Applications}
\label{sec:Appl}

\subsection{Bounds on $L$-functions}
As in~\cite{Gal,Iwan}, we can apply our bound on the sums $\TS_\chi(M,N;t)$ to 
estimate the size of $L$-functions inside the critical strip.

\begin{theorem}
\label{thm:|L(s,chi)|}
Fix $C>0$ and $\eta\in(0,\frac12)$. 
There is an effectively computable constant
$\gamma_0>0$ that depends only on $C$ and has the following property.
Let $q$ be a modulus satisfying~\eqref{eq:minmax}
and $\chi$ a primitive character modulo $q$.
If the inequalities $\sigma>1-\eta$
and $|t|\le q^C$ hold, then for $s=\sigma+it$ we have
$$
|L(s,\chi)|\le \eta^{-1}\exp\bigl(O\bigl(\max\bigl\{\eta\log\fq,\eta^{3/2}\ell,
\eta\ell^{2/3}(\log\ell)^{1/3}\bigr\}\bigr)\bigr),
$$
where $\ell=\log q(|t|+3)$, and the implied constant depends only on $C$.
\end{theorem}

To illustrate the strength of the bound,
we note that with the specific choice
$$
\eta = \frac{1}{\ell^{1/2}(\log\ell)^{3/4}},
$$
as considered by Iwaniec~\cite{Iwan}, our
Theorem~\ref{thm:|L(s,chi)|} yields the bound
$$
|L(s,\chi)| \le
\fq^{o(1)}\exp\bigl(O\bigl(\ell^{1/4}(\log \ell)^{-9/8}\bigr)\bigr)
$$
for $\sigma>1-\eta$
provided that $|t|$ is polynomially bounded in terms of $q$,
where $o(1)$ is a function that tends to zero as $q\to\infty$.
In particular, this improves the bound of~\cite[Theorem~1]{Iwan}, i.e.,
$$
|L(s,\chi)|\le \fq^{o(1)}\exp\bigl(100\ell^{1/4}\bigr),
$$
under the same condition on $t$ (however, the Iwaniec bound
also holds for all larger values of $t$).
It is important to note that for all known applications to
the distribution of primes, only values of $s=\sigma+it$
with $t$ growing as  a small power of $q$ (typically, $|t|\le q$)
play an important r\^ole; see Section~\ref{sec:psi int}
where we give one application of this type. 

Taking $\eta$ somewhat smaller, namely
$$
\eta  =  \frac{(\log\ell)^{2/3}}{\ell^{2/3}}
$$ 
(in other words, taking values of $s$ that lie even closer to the
edge of the critical strip), Theorem~\ref{thm:|L(s,chi)|} yields the bound 
$$
|L(s,\chi)|\le \fq^{o(1)}(\log q)^{O(1)}
$$
for $\sigma>1-\eta$ provided that $|t|$ is polynomially bounded in terms of $q$.

Choosing $\eta$ even smaller, namely 
$$
\eta  =  \frac{1}{\ell^{2/3}(\log\ell)^{1/3}}, 
$$
we obtain the following attractive bound
\begin{equation}
\label{eq:nice L}
L(s,\chi)\ll\fq^{o(1)}(\log q)^{2/3}(\log\log q)^{1/3}
\end{equation}
for $\sigma>1-\eta$ provided that $|t|$ is polynomially bounded in terms of $q$.
In particular, the bound~\eqref{eq:nice L} applies to
$L(1,\chi)$ and is therefore of special interest as it presently unknown whether
the estimate
$$
L(1,\chi)=o(\log q)
$$
holds for general moduli $q$ (although the bound $L(1,\chi)\ll\log\log q$
is implied by the GRH);
for the strongest unconditional upper bounds on $|L(1,\chi)|$,
see Granville and Soundararajan~\cite{GranSou}.

We conclude this subsection with the remark that, in our setting,
one can define $\ell$ more simply as $\ell=\log q$. In Theorem~\ref{thm:|L(s,chi)|}
and in the above examples, we have used the
definition $\ell=\log q(|t|+3)$ solely for the purpose 
of comparing our results to those of~\cite[Theorem~1]{Iwan}.

\subsection{The zero-free region}
\label{eq:zerofree}
We now apply our new bounds on $L$-functions to
extend the zero-free region on low-lying zeros. 
Note that we formulate the results of this section
only for primitive characters $\chi$ modulo 
$q$  satisfying~\eqref{eq:minmax}; for other characters,
our results can be
formulated in terms of the conductor of~$\chi$.

\begin{theorem}
\label{thm:zerofreeregion}
For every $C>0$, there is an effectively computable constant
$\gamma_0>0$ that depends only on $C$ and has the following property.
Let $q$ be a modulus satisfying~\eqref{eq:minmax}.
There is a constant $A>0$, which depends only on $C$ and
$\fq$, such that if
\begin{equation}
\label{eq:varthetashape}
\vartheta=\frac{A}{(\log q)^{2/3}(\log\log q)^{1/3}},
\end{equation}
then there exists at most one
primitive character $\chi$ modulo $q$
such that $L(s,\chi)$ has a zero in the region
$\bigl\{s\in\CC:\sigma>1-\vartheta,~|t|\le q^C\bigr\}$.
If such a character exists, then it is a real character, and
the zero is unique, real and simple.
\end{theorem}

It is worth mentioning that, under the same conditions
as in Theorem~\ref{thm:zerofreeregion}, 
the result of Iwaniec~\cite[Theorem~2]{Iwan}
 yields a similar bound  with
$(\log q(|t|+3))^{3/4}(\log\log q(|t|+3))^{3/4}$
in the denominator of
$\vartheta$ instead of $(\log q)^{2/3}(\log\log q)^{1/3}$,
but without any restriction on $|t|$.  Of course, for applications
to exceptional characters this restriction on $|t|$
is irrelevant, and thus Theorem~\ref{thm:zerofreeregion} eliminates a
substantial part of the real interval $[0,1]$ where a
\emph{Siegel zero} might possibly occur. 

\begin{corollary}
\label{cor:real char}
Let $q$ be a modulus satisfying~\eqref{eq:minmax}.
There is a constant $A>0$, which depends only on  
$\fq$, with the following property.
Let $\vartheta$ be given by \eqref{eq:varthetashape}. 
Then there exists at most one primitive real
character $\chi$ modulo $q$ such that $L(\sigma,\chi)$ has a
zero in the region $1 \ge \sigma > 1-\vartheta$, which in this
case is then a simple zero. 
\end{corollary}

We remark that, in the most interesting
case in which $q=p^\gamma$ is a power of a fixed prime $p$, 
the condition~\eqref{eq:minmax} is satisfied
automatically once $\gamma\ge\gamma_0$,
and thus Theorem~\ref{thm:zerofreeregion} yields the following
statement for all 
characters modulo an odd prime power $q=p^\gamma$.

\begin{corollary}
\label{cor:zerofreeregion}
Let $q=p^\gamma$ with $p$ an odd prime and $\gamma\in\NN$.
For every $C>0$, there is a constant $A>0$, which depends
only on $C$ and $p$, with the following property.
Let $\vartheta$ be given by \eqref{eq:varthetashape}.
For any character $\chi$
modulo $q$, the function $L(s,\chi)$ has no zero in the
region $\bigl\{s\in\CC:\sigma>1-\vartheta,~|t|\le q^C\bigr\}$.
\end{corollary}

\subsection{Primes in arithmetic progressions and short intervals}
\label{sec:psi int}

As usual we use $\Lambda$ to denote the von~Mangoldt
function, which is given by
$$
\Lambda(n)=
\begin{cases}
\log r &\quad\text{if $n$ is a power of the prime $r$,} \\
0&\quad\text{if $n$ is not a prime power,}
\end{cases}
$$
and we set
$$
\psi(x;q,a) = \sum_{\substack{n \le x\\ n \equiv a  \bmod q}} \Lambda(n).
$$
The asymptotic formula in Theorem~\ref{thm:prime AP} below
has a smaller error term than that which appears in any
other asymptotic formula of this type.
As in~\cite{Gal,Iwan} our result depends on 
density estimates for the zeros of Dirichlet $L$-functions. More specifically, 
let $N_q(\alpha,T)$ be the total number of zeros $s=\sigma+it$
for all \text{$L$-functions} modulo $q$ that occur in 
the rectangle $\alpha < \sigma < 1$, $|t| \le T$. 
In order to state a general result suitable for further advances, 
we assume that for some constant $b>1$ the uniform bound 
\begin{equation}
\label{eq:ZeroDen}
N_q(\alpha,T) \ll (qT)^{b(1-\alpha)} \ell^{O(1)}
\end{equation}
holds, where as before $\ell=\log q(|t|+3)$. By a
result of Huxley~\cite{Hux} we can take $b=12/5$ in
\eqref{eq:ZeroDen}; see also~\cite[Equation~(18.13)]{IwKow}.

\begin{theorem}
\label{thm:prime AP}
Suppose \eqref{eq:ZeroDen} holds with some constant $b>1$.
Fix an odd prime $p$ and a real number $\eps>0$. 
There is a constant $c_0>0$, which depends only on $b$, $\eps$ and $p$,
such that the following holds.
For any modulus $q=p^\gamma$ with $\gamma\in\NN$, any integer $a$ 
coprime to $p$, and any positive real numbers $x$ and $h$ for which
$$
q x^{1-1/b+\eps}\le h\le x\le q^{1/\eps},
$$
we have 
$$
\psi(x+h;q,a)-\psi(x;q,a)=\frac{h}{\varphi(q)}
+O_\eps(h\exp(-c_0(\log x)^{1/3}(\log\log x)^{-1/3})),
$$
where $\varphi$ is the Euler totient function. 
\end{theorem}

In particular, using the value $b=12/5$
we see that Theorem~\ref{thm:prime AP} can be applied throughout
the range $q^A \ge x \ge h \ge q x^{7/12+\eps}$.

Our proof of Theorem~\ref{thm:prime AP} closely follows that
of Gallagher~\cite[Theorem~2]{Gal}, however we apply
Corollary~\ref{cor:zerofreeregion} at an appropriate place.
We remark that the results of
Gallagher~\cite{Gal} and Iwaniec~\cite{Iwan}  imply only a weaker form of  
Theorem~\ref{thm:prime AP} with the   error term 
$O(h\exp(-c_0(\log x)^{1/4}(\log\log x)^{3/4}))$.

\section{Preliminaries}

\subsection{Notation}

For a real number $t>0$, $\fl{t}$ denotes the greatest integer
not exceeding $t$, and $\rf{t}$ denotes
the least integer that is not less than $t$.

Throughout the paper, we use the symbols $O$, $\ll$, $\gg$
and $\asymp$ along with their standard meanings; any constants
or functions implied by these symbols are \emph{absolute} unless
specified otherwise.

\subsection{Polynomial representation of characters}

Following Gallagher~\cite{Gal}, for an integer $d\ge 1$ we use 
$F_d$ to denote the polynomial approximation to $\log(1+x)$ given by
\begin{equation}
\label{eq:Fdxdefn}
F_d(x)= \sum_{r=1}^d  (-1)^{r-1}\frac{x^r}{r}.
\end{equation}
According to~\cite[Lemma~2]{Iwan} (which extends~\cite[Lemma~2]{Gal})
we have the following statement.

\begin{lemma}
\label{lem:iwaniec}
Let $\chi$ be a primitive character modulo $q$.
Let $d$ be an integer such that $q^2\mid \fq^d$,
and put
$$
\tau=\begin{cases}
2&\quad\hbox{if $4\mid q$,}\\
1&\quad\hbox{otherwise}.
\end{cases}
$$
Then $\chi(1+\tau \fq x)=\e(f(x))$, where $f$ is a polynomial of the form
$$
f(x)= q^{-1}m\cdot F_d(\tau \fq x)
$$
with an integer $m$ for which $\gcd(m,q)=1$, and $r\mid m$ for every
integer $r\in[1,d]$ coprime to $q$.
\end{lemma}

\subsection{Bounds of exponential sums}

Suppose $d\ge 2$, and $g(x)=\alpha_1 x+\cdots+\alpha_d x^d$
with each $\alpha_r\in\RR$.  Suppose further that
each $\alpha_r$ has a rational approximation of the form
$$
\alpha_r=\frac{a_r}{b_r}+\frac{\vartheta_r}{b_r^2},\qquad
a_r\in\ZZ,\qquad b_r\in\NN,\qquad
\gcd(a_r,b_r)=1,\qquad
|\vartheta_r|\le 1.
$$
Let $S$ denote the double exponential sum
\begin{equation}
\label{eq:doublesumdefn}
S=\sum_{y,z=1}^P \e(g(yz)).
\end{equation}
The next result is due to Korobov~\cite[Lemma~3]{Kor}; it
provides a bound on $S$ in terms of $N_{k,d}(P)$
(the number of solutions to~\eqref{eq:System}) and a
product involving the denominators of the coefficients of $g$. 

\begin{lemma}
\label{lem:korobov}
For any natural number $k$, the sum~\eqref{eq:doublesumdefn}
admits the upper bound
$$
|S|^{2k^2}\le \(64k^2\log(3Q)\)^{d/2}
WP^{2k(2k-1)} N_{k,d}(P),
$$
where 
$$
Q=\max\{b_r:1\le r\le d\}\mand
W=\prod_{r=1}^d\min\left\{P^r,P^rb_r^{-1/2}+b_r^{1/2}\right\}.
$$
\end{lemma}
 
We also use the following weakened and simplified
version of a result of Ford~\cite[Theorem~3]{Ford}.

\begin{lemma}
\label{lem:ford}
For every integer $d\ge 129$ there is an integer
$k\in[2d^2,4d^2]$ such that
$$
N_{k,d}(P)\le d^{3d^3}P^{2k-0.499d^2}\qquad(P\ge 1).
$$
\end{lemma}

\section{Proof of bounds of character sums}

\subsection{Simple character sums: Proof of Theorem~\ref{thm:mainG}}
\label{sec:proof-mainG}

Let $\gamma_0$ and $\eps$ be positive constants such that
\begin{equation}
\label{eq:cond}
\gamma_0\ge e^{200},\qquad
\eps\le 1/200\mand \eps\gamma_0\ge 2.
\end{equation}
Put $d_0=2\gamma$.
Since $\gamma=\max_{p\mid q}\{v_p(q)\}$,
the condition $q^2\mid \fq^{d_0}$ of Lemma~\ref{lem:iwaniec} is clearly met.
Also, the parameter $\varrho$ lies in $[1,\gamma/\gamma_0]$ since
$$
\log q=\sum_{p\mid q}v_p(q)\log p\le\gamma\sum_{p\mid q}\log p=\gamma\log \fq,
$$
whereas by~\eqref{eq:condsC} we have
$$
\log N\ge \gamma_0\log\fq.
$$
Put $s=\fl{\eps\gamma/\varrho}$.
Since $\eps\gamma/\varrho\ge\eps\gamma_0\ge 2$,  it follows that
\begin{equation}
\label{eq:halfepsgamma}
\tfrac12\eps\gamma/\varrho\le \eps\gamma/\varrho-1<s\le \eps\gamma/\varrho,
\end{equation}
and thus $s\asymp\gamma/\varrho$.  Using~\eqref{eq:cond} and~\eqref{eq:halfepsgamma} we deduce that
\begin{equation}
\label{eq:del_s_hyp}
\log\gamma\ge 200
\mand 2\le s\le\gamma/200.
\end{equation}
Finally, we record the simple inequality
\begin{equation}
\label{eq:del_s_hyp2}
t\le e^{t/1250}\qquad(t\ge \gamma_0).
\end{equation}

Let $\cN$ be the set of integers coprime to $q$ in the interval
$[M+1,M+N]$. Then,
\begin{equation}
\label{eq:library}
S_\chi(M,N;G) = \fq^{-2s}V+O(\fq^{3s}),
\end{equation}
where
$$
V=\sum_{y,z=1}^{\fq^s}\sum_{n\in\cN}\chi(n+\fq^s yz) \e(H_n(yz))
$$
and $H_n$ is the polynomial given by
$$
H_n(x) = G(n+\fq^s x)).
$$
For every $n\in\cN$, let $\overline n$ be an integer such that 
$n\overline n\equiv 1\bmod{q}$. Using the multiplicativity 
of $\chi$, we have 
$$
V =\sum_{n\in\cN}\chi(n)\sum_{y,z=1}^{\fq^s}\chi(1+\fq^s\overline nyz) \e(H_n(yz)).
$$
Applying Lemma~\ref{lem:iwaniec} (noting that $s\ge 2$ and
thus $\tau \fq\mid\fq^s$) we see that
\begin{equation}
\label{eq:Vequals}
V=\sum_{n\in\cN}\chi(n)\sum_{y,z=1}^{\fq^s}\e(f_n(yz)+ H_n(yz)),
\end{equation}
where $f_n$ is a polynomial of the form
$$
f_n(x)= q^{-1}m\cdot F_{d_0}(\fq^s\overline nx)
$$
with some integer $m$ such that $\gcd(m,q)=1$ and $r\mid m$ for any
integer $r\in[1,d_0]$ coprime to $q$.
To apply Lemma~\ref{lem:korobov}, we need to 
control the denominators of  the coefficients
of $f_n + H_n$ for each $n\in\cN$. 

Using~\eqref{eq:Fdxdefn} we see that the
$r$-th coefficient of $f_n$ is the rational number
$$
\alpha_r=(-1)^{r-1}\fq^{rs}q^{-1}m\overline n^r r^{-1}.
$$
Write
$$
\alpha_r=\frac{a_r}{b_r},\qquad
a_r\in\ZZ,\qquad b_r\in\NN,\qquad
\gcd(a_r,b_r)=1\qquad(1\le r\le d_0).
$$
Since $r\mid m$ for every
integer $r\in[1,d_0]$ coprime to $q$,
and $\gcd(m\overline n,q)=1$, it follows that $b_r$
is the numerator of the rational number
$$
q\fq^{-rs}\prod_{p\mid\gcd(r,q)}p^{v_p(r)}
$$
when the latter is expressed in reduced form (in particular,
 $b_r$ is composed solely of primes that divide $q$).
Consequently,
$$
v_p(b_r)=\max\{0,v_p(q)-rs+v_p(r)\}
$$
for every prime $p$ dividing $q$.

Let us denote
$$
\sL=\fl{\tfrac32\log d_0}=\fl{\tfrac32\log 2\gamma}.
$$
As the inequality
$v_p(r)\le \sL$ holds for every positive integer $r\le d_0$, we have
\begin{equation}
\label{eq:vpqrx}
\max\{0,v_p(q)-rs\}\le v_p(b_r)\le\max\{0, v_p(q)-rs+\sL\}
\end{equation}
for any prime $p\mid q$.

Now put
\begin{equation}
\label{eq:kappa_s_hyp}
d=\max\limits_{p\mid q}\fl{\frac{v_p(q)+\sL}{s}}
=\fl{\frac{\gamma+\sL}{s}}.
\end{equation}
Note that $d\ge 200$ since $\gamma/s\ge 200$ by~\eqref{eq:del_s_hyp};
in particular, we are able to apply
Lemma~\ref{lem:ford} below with this choice of $d$. 

For any integer $r\ge d$,
it follows from~\eqref{eq:vpqrx} that $b_r=1$; in other words,
$\alpha_r\in\ZZ$. Therefore, defining
$$
g_n(x)=q^{-1}m\cdot F_d(\fq^s\overline nx)\qquad(n\in\cN),
$$
the polynomial $f_n-g_n$ lies in $\ZZ[x]$ for every $n\in\cN$;
therefore, in view of~\eqref{eq:Vequals} we have
\begin{equation}
\label{eq:Vequals2}
V=\sum_{n\in\cN}\chi(n)\sum_{y,z=1}^{\fq^s}\e(h_n(yz)),
\end{equation}
where
$$
h_n(x) = g_n(x)+H_n(x).
$$

Suppose that $\eps$ is initially chosen to be small enough, depending on $C$, 
so that $C\le(3\eps)^{-1}$.  In view of~\eqref{eq:halfepsgamma},
the second inequality in~\eqref{eq:condsC} implies 
\begin{equation}
\label{eq:deg G}
 \deg G\le\gamma/(3s).
\end{equation}
We now use approximations with denominators $\fb_r =1$ for
the initial $\fl{\gamma/(3s)}$ coefficients of $h_n$
(i.e., for $1 \le r \le \gamma/(3s)$)
and with the denominators $\fb_r = b_r$ considered above
for remaining coefficients of $h_n$
 (i.e., for $r > \gamma/(3s)$), which by~\eqref{eq:deg G} 
are the same as the coefficients of  $g_n(x)$. 

Put
$$
Q=\max\{\fb_r:1\le r\le d\}\mand
W=\prod_{r=1}^d\min\left\{\fq^{rs},\fq^{rs}\fb_r^{-1/2}+\fb_r^{1/2}\right\},
$$
Applying Lemma~\ref{lem:korobov} with $P=\fq^s$ we derive the bound
\begin{equation}
\label{eq:kindle1}
\left|\sum_{y,z=1}^{\fq^s}\e(h_n(yz))\right|^{2k^2}
\le \(64k^2\log(3Q)\)^{d/2} W\fq^{2sk(2k-1)} N_{k,d}(\fq^s)
\end{equation}
with any natural number $k$.  Using~\eqref{eq:vpqrx} we have that
\begin{equation}
\label{eq:painting}
q\fq^{-rs}\le b_r\le q\fq^{-rs+\sL}
\qquad(1\le r\le d).
\end{equation}
In particular, $Q\le q\fq^\sL$, which implies (since $q\le\fq^\gamma$)
\begin{equation}
\label{eq:kindle2}
\log(3Q)\le 2\gamma\log \fq.
\end{equation}
Next, note that the hypothesis~\eqref{eq:minmax} immediately yields the bound
$$
\log q=\sum_{p\mid q}v_p(q)\log p\ge0.7\gamma\sum_{p\mid q}\log p=0.7\gamma\log \fq,
$$
hence $q=\fq^{\mu\gamma}$ with some $\mu\in[0.7,1]$.
To estimate $W$, we use~\eqref{eq:painting} to derive the bound
$$
\min\left\{\fq^{rs},\fq^{rs}\fb_r^{-1/2}+\fb_r^{1/2}\right\}\le\begin{cases}
\fq^{rs}&\quad\hbox{if $r\le\gamma/(3s)$};\\
2\fq^{(\mu\gamma-rs+\sL)/2}&\quad\hbox{if $\gamma/(3s)<r\le \mu\gamma/(2s)$};\\
2\fq^{(3rs-\mu\gamma)/2}&\quad\hbox{if $\mu\gamma/(2s)<r\le \gamma/s$};\\
\fq^{rs}&\quad\hbox{if $\gamma/s<r\le d$}.
\end{cases}
$$
To simplify the notation, let $\lambda=\gamma/s$ for the moment. 
Using the preceding bound, we have
$$
W\le
\prod_{r\le \lambda/3}\fq^{rs}
\prod_{\lambda/3<r\le \mu\lambda/2}\(2\fq^{(\mu\gamma-rs+\sL)/2}\)
\prod_{\mu\lambda/2<r\le \lambda}\(2\fq^{(3rs-\mu\gamma)/2}\)
\prod_{\lambda<r\le d}\fq^{rs}\le 2^d \fq^\Delta,
$$
where
$$
\Delta=
\sum_{r\le \lambda/3}rs+
\sum_{\lambda/3<r\le \mu\lambda/2}\frac{\mu\gamma-rs+\sL}{2}+
\sum_{\mu\lambda/2<r\le \lambda}\frac{3rs-\mu\gamma}{2}+
\sum_{\lambda<r\le d}rs.
$$
We write 
\begin{equation}
\label{eq:Delta}
\begin{split}
\Delta&=s\Sigma+\frac{\mu\gamma}{2}\(\frac{\mu\lambda}{2}-\frac{\lambda}{3}+O(1)\)
-\frac{\mu\gamma}{2}\(\lambda-\frac{\mu\lambda}{2}+O(1)\)+O(\sL\lambda)\\
&=s\Sigma+\mu\gamma\(\frac{\mu\lambda}{2}-\frac{2\lambda}{3}\)+O(\gamma+\sL\lambda)
\end{split}
\end{equation}
(recall our convention that all implied constants are absolute), with 
\begin{align*}
\Sigma&=
\sum_{r\le\lambda/3}r-
\frac{1}{2}\sum_{\lambda/3<r\le\mu\lambda/2}r+
\frac{3}{2}\sum_{\mu\lambda/2<r\le\lambda}r+\sum_{\lambda<r\le d}r\\
&\le\frac{1}{2}\(\frac{\lambda}{3}\)^2
-\frac{1}{4}\(\(\frac{\mu\lambda}{2}\)^2-\(\frac{\lambda}{3}\)^2\)
+\frac{3}{4}\(\lambda^2-\(\frac{\mu\lambda}{2}\)^2\)
+\frac{1}{2}\(d^2-\lambda^2\)+O(d).
\end{align*}
Since $d=\lambda+O(1)$ and thus
$d^2-\lambda^2=O(\lambda)$, we derive that
$$
\Sigma=
\(\frac{5}{6}-\frac{\mu^2}{4}\)\lambda^2+O(\lambda).
$$
Inserting this result into~\eqref{eq:Delta}, recalling that
$\lambda=\gamma/s$ and $\mu\in[0.7,1]$, and using~\eqref{eq:kappa_s_hyp}, it follows that
$$
\Delta=\(\frac56+\frac{\mu^2}{4}-\frac{2\mu}{3}\)\frac{\gamma^2}{s}
+O\(\gamma+\frac{\gamma\sL}{s}\)
\le 0.49sd^2+O(sd\log d).
$$
Therefore, if $\eps$ is small enough initially (depending on the absolute implied 
constant in the preceding bound), then we have
$$
\Delta  \le  0.495 sd^2,
$$
and thus
\begin{equation}
\label{eq:kindle3}
W\le 2^d p^{0.495 sd^2}.
\end{equation}
Now, combining the bounds~\eqref{eq:kindle1}, \eqref{eq:kindle2} and~\eqref{eq:kindle3},
and using Lemma~\ref{lem:ford} to bound $N_{k,d}(\fq^s)$, we deduce that
$$
\left|\sum_{y,z=1}^{\fq^s}\e(h_n(yz))\right|^{2k^2}\le  A\fq^B
$$
holds with
$$
A=\(128k^2\gamma\log \fq\)^{d/2} 2^d d^{3d^3}
$$
and
$$
B=4sk^2-0.004sd^2
$$
for some integer $k\in[2d^2,4d^2]$.

Since $k\in[2d^2,4d^2]$ we clearly
have $A\le d^{c d^5}(\gamma\log \fq)^{d/2}$ with
some absolute (effective) constant $c >0$.  As $\gamma\log \fq\ge\gamma_0$,
using~\eqref{eq:del_s_hyp2} and taking into account the definition
\eqref{eq:kappa_s_hyp}, which implies that $\gamma\le 2sd$,
it follows that
$$
(\gamma\log \fq)^{d/2}\le \fq^{0.0004\gamma d}\le \fq^{0.0008sd^2}.
$$
Putting everything together, we find that
$$
\left|\sum_{y,z=1}^{\fq^s}\e(h_n(yz))\right|^{2k^2}\le
d^{cd^3}\fq^{4sk^2-0.0032sd^2}.
$$
Raise both sides to the power $1/(2k^2)$.  Since $k\in[2d^2,4d^2]$
we have
$$
d^{cd^3/(2k^2)} \le d^{c/(8d)} \ll 1 \mand 
sd^2/(2k^2) \ge s/(32d^2); 
$$
consequently,
$$
\sum_{y,z=1}^{\fq^s}\e(h_n(yz))\ll \fq^{2s-0.0001s/d^2}.
$$
Finally, using~\eqref{eq:halfepsgamma} and~\eqref{eq:del_s_hyp}
we see that
$$
\frac{s}{d^2}\asymp\frac{s}{(\gamma/s)^2}
=\frac{s^3}{\gamma^2}\asymp\frac{(\gamma/\varrho)^3}{\gamma^2}
=\frac{\gamma}{\varrho^3}\asymp\frac{\mu\gamma}{\varrho^3},
$$
and therefore
$$
\sum_{y,z=1}^{\fq^s}\e(g_n(yz))\ll \fq^{2s-\xi_0\mu\gamma/\varrho^3}
=\fq^{2s}N^{-\xi_0/\varrho^2}
$$
with some absolute constant $\xi_0>0$.

Inserting the previous bound into~\eqref{eq:Vequals2} we derive that
$$
V\ll \fq^{2s}N^{1-\xi_0/\varrho^2}
$$
and combining this result with~\eqref{eq:library} we obtain that
\begin{equation}
\label{eq:S almost}
S_\chi(M,N;G) \ll N^{1-\xi_0/\varrho^2}+\fq^{3s}.
\end{equation}
The second term on the right side of~\eqref{eq:S almost} is negligible (indeed, 
using~\eqref{eq:halfepsgamma} we have $\fq^s\le N^{\eps/\mu}$, hence
$\fq^{3s}\le N^{5\eps}$, which is insignificant
compared to $N^{1-\xi_0/\varrho^2}$ if one makes suitable initial
choices of the absolute constants $\gamma_0$, $\xi_0$ and $\eps$).
This completes the proof.

\subsection{Dirichlet polynomials: Proof of Theorem~\ref{thm:mainT}}

We continue to use the notation of \S\ref{sec:proof-mainG}.
We denote $\nu=\rf{\gamma/(3s)}$.
For any real number $x$, we have the estimate
$$
(1+x)^{it} = \e(tG(x))\(1+O(|t||x|^{\nu})\), 
$$
where $G(x)=(2\pi)^{-1}F_{\nu-1}(x)$ in the notation of~\eqref{eq:Fdxdefn}
(note that $G(x)$ is a polynomial of degree $\nu-1$ with real coefficients).
Hence, for all $n\in[M+1,M+N]$ and $y,z\in [1,\fq^s]$ we have
$$
\(n+\fq^syz\)^{it}=n^{it}(1+\fq^syz/n)^{it}
=n^{it}\e(tG(\fq^syz/n))+O(N^{-\nu}|t|\fq^{3s\nu}).
$$
Here we have used the fact that $M\asymp N$.
Using this estimate and following the proof of Theorem~\ref{thm:mainG},
in place of~\eqref{eq:library} we derive that 
$$
\TS_\chi(M,N;t)= \fq^{-2s}\tV+O(\fq^{3s} +N^{1-\nu}|t|\fq^{3s\nu}),
$$
where
$$
\tV =\sum_{n\in\cN}\chi(n)n^{it}\sum_{y,z=1}^{\fq^s}\chi(1+\fq^s\overline nyz)
\e(tG(\fq^syz/n)). 
$$
Since $ \deg G<\gamma/(3s)$, at this point the proof
parallels that of Theorem~\ref{thm:mainG}, leading to the bound
\begin{equation}
\label{eq:St almost}
\TS_\chi(M,N;t)\ll N^{1-\xi_0/\varrho^2}+ \fq^{3s} +N^{1-\nu}|t|\fq^{3s\nu}
\end{equation}
in place of~\eqref{eq:S almost}. As before, the term
$\fq^{3s}$ in~\eqref{eq:St almost} does not exceed $N^{5\eps}$
and can thus be disregarded if one makes
suitable initial choices of $\gamma_0$, $\xi_0$ and $\eps$.

To finish the proof, it remains to bound the last term in~\eqref{eq:St almost}.
Let $\tau$ be such that $N^\tau=|t|+3$.
Since $\nu=\rf{\gamma/(3s)}$, it follows that $3s\nu\le\gamma+3s$,
and by~\eqref{eq:halfepsgamma} we have
$\nu\ge\gamma/(3s)\ge\varrho/(3\eps)$; therefore, 
$$
N^{1-\nu}|t|\fq^{3s\nu}
\ll N^{1-\varrho/(3\eps)+\tau}\fq^{\gamma+3s}.
$$
We have $\fq^{3s}\le N^{5\eps}$ as before, and by~\eqref{eq:minmax}
it follows that $\fq^\gamma\le N^{2\varrho}$.
We get that
$$
N^{1-\nu}|t|\fq^{3s\nu}
\ll N^{1-\varrho/(3\eps)+\tau+2\varrho+3\eps}.
$$
Inserting this bound into~\eqref{eq:St almost},
the theorem is a consequence of the inequality
$$
\tau\le\varrho((3\eps)^{-1}-2)-\xi_0/\varrho^2-3\eps,
$$
which follows from the last inequality in~\eqref{eq:condsT}
(which implies, $\tau\le C\varrho + o(1)$) 
assuming that $\eps$ and $\xi_0$
are sufficiently small in terms of $C$.

\section{Proofs of results for $L$-functions and distribution of primes in progressions}

\subsection{Bounds on $L$-functions and zero-free regions: Proof of Theorem~\ref{thm:|L(s,chi)|}}

We begin with a general statement involving
two parameters $\eta$ and $Y$. 

\begin{lemma}
\label{lem:mainL}
For any real number $C>0$ there are effectively computable constants
$\gamma_0,\xi_0,c_0>0$ that depend only
on $C$ and have the following property. Let $q$ be a
modulus satisfying~\eqref{eq:minmax}
and $\chi$ a primitive character modulo $q$.
If $Y$ and $\eta$ satisfy
\begin{equation}
\label{eq:piano}
Y\ge \fq^{\gamma_0}\mand
0<\eta\le\xi_0(\log Y)^2/\ell^2-c_0(\log\ell)/\log Y, 
\end{equation}
where $\ell=\log q(|t|+3)$,
and the inequalities $\sigma>1-\eta$ and $|t|\le q^C$ hold, then
for $s = \sigma + it$ we have
$$
|L(s,\chi)|\le \eta^{-1}Y^\eta.
$$
\end{lemma}

\begin{proof}
Fix $C>0$, and let $\gamma_0,\xi_0>0$
have the property described in Theorem~\ref{thm:mainT}.  Let $q$ be a
modulus satisfying~\eqref{eq:minmax}
and $\chi$ a primitive character modulo $q$.
By Theorem~\ref{thm:mainT} and partial summation, the bound
\begin{equation}
\label{eq:trump}
\sum_{N<n\le 2N}\chi(n)n^{-s}\ll N^{1-\sigma-\xi_0/\varrho^2}
\qquad(N\ge \fq^{\gamma_0})
\end{equation}
holds, where $\varrho=(\log q)/\log N$ and the implied constant depends only on $C$.

Put $Z=e^{2\ell}$.
Arguing as in the proof of~\cite[Lemma~8]{Iwan}, the bound
\begin{equation}
\label{eq:Zbound}
\left|\sum_{n>Z}\chi(n)n^{-s}\right|\le 1
\end{equation}
holds since $\sigma>\tfrac12$.
On the other hand, let $Y$ and $\eta$ be real
numbers such that satisfy~\eqref{eq:piano}
with some constant $c_0>0$ that depends only $C$.
Assuming that $\sigma>1-\eta$, the bounds~\eqref{eq:piano} and~\eqref{eq:trump} 
imply
$$
\sum_{N<n\le 2N}\chi(n)n^{-s}\ll N^{\eta-\xi_0/\varrho^2}\le Y^{\eta-\xi_0/\varrho^2}
\le\ell^{-c_0}\qquad(N\ge Y).
$$
Hence, if $c_0$ is sufficiently large in terms of $C$, then for $\sigma>1-\eta$
we have
$$
\left|\sum_{N<n\le 2N}\chi(n)n^{-s}\right|\le (3\ell)^{-1}\qquad(N\ge Y),
$$
which by a standard splitting argument yields the bound
$$
\left|\sum_{n\le Z}\chi(n)n^{-s}\right|
\le 1+\left|\sum_{n\le Y}\chi(n)n^{-s}\right|
\le 1+\sum_{n\le Y}n^{\eta-1}
\le 2+\eta^{-1}(Y^\eta-1).
$$
Combining this bound with~\eqref{eq:Zbound}, and assuming that $\eta\le\tfrac13$,
it follows that
$$
|L(s,\chi)|\le\eta^{-1}Y^\eta
$$
provided that $\sigma>1-\eta$.
\end{proof}

We now turn to the proof of Theorem~\ref{thm:|L(s,chi)|}.
Let the notation be as in Lemma~\ref{lem:mainL}.
The first inequality in~\eqref{eq:piano} is
\begin{equation}
\label{eq:ants}
\log Y\ge\gamma_0\log\fq.
\end{equation}
If $Y$ also satisfies the inequality
\begin{equation}
\label{eq:large Y 1}
\log Y \ge (2c_0/\xi_0)^{1/3}  \ell ^{2/3} (\log \ell)^{1/3},
\end{equation}
then it follows that
$$
\xi_0(\log Y)^2/\ell^2-c_0(\log\ell)/\log Y
\ge 0.5\xi_0(\log Y)^2/\ell^2;
$$
hence the second inequality in~\eqref{eq:piano} is
satisfied provided that the lower bound
\begin{equation}
\label{eq:large Y 2}
\log Y\ge 2^{1/2}\xi_0^{-1/2}\eta^{1/2}\ell
\end{equation}
also holds. Consequently, defining $Y$ by the equation  
$$
\log Y=A\max\bigl\{\log\fq,\eta^{1/2}\ell,\ell^{2/3}(\log\ell)^{1/3}\bigr\}
$$
with a suitably large absolute constant $A>0$
(depending only on $\gamma_0,\xi_0,c_0$),
we see that the inequalities
\eqref{eq:ants}, \eqref{eq:large Y 1} and~\eqref{eq:large Y 2} all hold,
hence the condition~\eqref{eq:piano} is met.
Applying Lemma~\ref{lem:mainL} we obtain the stated bound. 

\subsection{The zero-free region: Proof of Theorem~\ref{thm:zerofreeregion}}

We start with a technical
result contained in Iwaniec~\cite{Iwan}, which we present in a generic form
suitable for further applications. 

\begin{lemma}
\label{lem:technicallemma}
Let $q$ be a fixed modulus. Let $\eta\in(0,\frac12)$, $T\ge 1$ and $M\ge e$
be numbers that can depend on $q$. Put
\begin{equation}
\label{eq:varthetadefn}
\vartheta=\frac{\eta}{400\log M},
\end{equation}
and suppose that
\begin{equation}
\label{eq:etacond}
\eta\log(5\log 3q)\le 3\log(2.5\vartheta).
\end{equation}
Suppose that $|L(s,\chi)|\le M$ for all primitive characters $\chi$
modulo~$q$ and all $s$ in the region
$\bigl\{s\in\CC:\sigma>1-\eta,~|t|\le 3T\bigr\}$. 
There is at most one primitive character $\chi$ modulo $q$
such that $L(s,\chi)$ has a zero in the region
$\bigl\{s\in\CC:\sigma>1-\vartheta,~|t|\le T\bigr\}$.
If such a character exists, then it is a real character, and
the zero is unique, real and simple.
\end{lemma}

\begin{proof}
The first part of the proof of~\cite[Lemma~11]{Iwan}
shows that $L(s,\chi)\ne 0$ throughout the region
$$
\Gamma=\begin{cases}
\{s\in\CC:\sigma>1-\vartheta,~|t|\le T\}&\quad\hbox{if $\chi^2\ne\chi_0$},\\
\{s\in\CC:\sigma>1-\vartheta,~\eta/4<|t|\le T\}&\quad\hbox{if $\chi^2=\chi_0$},\\
\end{cases}
$$
provided that
$$
6\log(5\log 3q)+\frac{16}{\eta}\log(M/5\vartheta)+\frac{8}{\eta}\log(2M/5\vartheta)
\le\frac{1}{15\vartheta},
$$
and this inequality is a consequence of~\eqref{eq:etacond} and the fact that
$$
\frac{24}{\eta}\log M=\frac{3}{50\vartheta}<\frac{1}{15\vartheta}.
$$
The second part of the proof of~\cite[Lemma~11]{Iwan}
then shows that if $L(s,\chi)=0$ for some $s$ in the region
$\{s\in\CC:\sigma>1-\vartheta,~|t|\le\eta/4\}$, then the zero is unique,
real and simple provided that
\begin{equation}
\label{eq:1000}
8\log(5\log 3q)+\frac{16}{\eta}\log(M/5\vartheta)\le\frac{1}{15\vartheta},
\end{equation}
and this inequality is a consequence of~\eqref{eq:etacond} and the fact that
$$
\frac{16}{\eta}\log M=\frac{1}{25\vartheta}<\frac{1}{15\vartheta}.
$$
Finally,~\cite[Lemma~12]{Iwan} shows that there is at most one nonprincipal character
$\chi$ modulo $q$ for which $L(s,\chi)$ has a real zero
$\beta>1-\vartheta$, provided that
$$
2\log(5\log 3q)+\frac{12}{\eta}\log(M/5\vartheta)\le\frac{2}{15\vartheta},
$$
which is consequence of~\eqref{eq:1000}.  The result now follows.
\end{proof}

Turning now to the proof of Theorem~\ref{thm:zerofreeregion}, we note that
with the choice
$$
\eta=\frac{(\log\log q)^{2/3}}{(\log q)^{2/3}}
$$ 
Theorem~\ref{thm:|L(s,chi)|} shows that $|L(s,\chi)|\le M$ for all primitive characters $\chi$ modulo~$q$ and all $s$ in the region
$\bigl\{s\in\CC:\sigma>1-\eta,~|t|\le 3q^{C}\bigr\}$, where
$$
M=(\log q)^B
$$
for some constant $B$ that depends only on $C$ and $\fq$.
Using~\eqref{eq:varthetadefn} to define $\vartheta$, we 
obtain~\eqref{eq:varthetashape} with $A=1/(400B)$.  Taking $B$ larger
(and $A$ smaller) if necessary, we can guarantee that $M\ge e$ and that the
condition~\eqref{eq:etacond} is met.  Applying Lemma~\ref{lem:technicallemma},
we obtain the statement of Theorem~\ref{thm:zerofreeregion}.

\subsection{The zero-free region: Proof of Corollary~\ref{cor:zerofreeregion}}

To prove Corollary~\ref{cor:zerofreeregion} we consider
only those moduli $q$ of the form $q=p^\gamma$, where $p$ is
a fixed odd prime and $\gamma\in\NN$; note that $\fq=p$ for all
such moduli.  Let $\gamma_0$, $A$ and $\vartheta$ be the
numbers supplied by Theorem~\ref{thm:zerofreeregion}
with the constant $C>0$; we can clearly assume that $\gamma_0\ge 2$.

With $p$ fixed, there are only finitely many primitive
characters $\chi$ of conductor $p^\gamma$ with $\gamma<\gamma_0$.
Consequently, after replacing
$A$ with a smaller number (which depends only on $C$ and $p$), 
we can guarantee that $L(s,\chi)$ does not vanish in the
region $\cR=\bigl\{s\in\CC:\sigma>1-\vartheta,~|t|\le q^C\bigr\}$
for any such primitive character.

Given an arbitrary character $\chi$ modulo $q=p^\gamma$, let
$q^*=p^{\gamma^*}$ be its conductor.  Since $\gcd(n,q)=1$
if and only if $\gcd(n,q^*)=1$, we see that $\chi$ is primitive
when viewed as a character modulo $q^*$.

If $\gamma^*<\gamma_0$, $L(s,\chi)$ does not vanish
in $\cR$ by our choice of $A$.
In particular, \emph{this holds true if $\chi$ is a real character}.
Indeed, if $\chi$ is real, then $\chi$ is either the principal
character modulo $p$ or the Legendre symbol modulo $p$, since
$(\ZZ/p^\gamma\ZZ)^\times$ is cyclic for odd $p$ and therefore admits
only two real characters.

If $\gamma^*\ge\gamma_0$ and $\chi$ is not real, then $L(s,\chi)\ne 0$
in $\cR$ by Theorem~\ref{thm:zerofreeregion}.

\subsection{Primes in arithmetic progressions: Proof of Theorem~\ref{thm:prime AP}}

As in the proof of~\cite[Theorem~2]{Gal} we define $T$ by the equation
$(qT)^b = x^{1-\eps}$. Since $x\le q^{1/\eps}$ we have
$$
T=q^{-1}x^{(1-\eps)/b}\le q^C\qquad\text{with}\quad C=\frac{1-\eps}{b\,\eps}-1;
$$
we can assume $\eps$ is small enough so that $C>0$.  Moreover,
since $q\le x^{1/b-\eps}$ we see that
$$
T=q^{-1}x^{1/b-\eps/b}\ge x^{\eps(1-1/b)}\ge 2
$$
if $x$ is large, which we can assume.

Since $\log q\asymp\log x$ holds with implied constants that depend only on $b$ and~$\eps$,
an application of Corollary~\ref{cor:zerofreeregion} shows that
there is a constant $a>0$ depending only on $b$, $\eps$ and $p$
such that $N_q(\alpha,T)=0$ for all $\alpha\ge\vartheta$, where 
\begin{equation}
\label{eq:rho}
\vartheta=\frac{a}{(\log x)^{2/3}(\log\log x)^{1/3}}.
\end{equation}

Using~\eqref{eq:ZeroDen} together with the ``trivial'' bound (see~\cite[Theorem~5.24]{IwKow})
$$
N_q(\alpha,T) \ll qT \ell,
$$
the first double sum in~\cite[Equation~(16)]{Gal} is bounded by
the following precise version of~\cite[Equation~(17)]{Gal}:
\begin{align*}
\int_0^{1-\vartheta} x^{\alpha-1} &N_q(\alpha,T)(\log x)\,d\alpha + x^{-1} N_q(0,T) \\
&\ll (\log x)^{O(1)} \int_0^{1-\vartheta} ((qT)^bx^{-1})^{1-\alpha}\,d\alpha +qTx^{-1}(\log x)^{O(1)}\\
&= (\log x)^{O(1)} \int_0^{1-\vartheta} x^{-\eps(1-\alpha)}\,d\alpha  + x^{(1-\eps)/b-1}  (\log x)^{O(1)}\\
&\ll_\eps x^{-\eps \vartheta} (\log x)^{O(1)}  + x^{(1-\eps)/b-1}  (\log x)^{O(1)},
\end{align*}
where the symbol $\ll_\eps$ indicates that the implied constant may depend on $\eps$.

Since $b>1$ implies that $(1-\eps)/b-1<0$, using \eqref{eq:rho} we see that the first term in the 
preceding bound dominates, and so we obtain that
$$
\int_0^{1-\vartheta} x^{\alpha - 1} N_q(\alpha,T)  \log x d\alpha + x^{-1} N_q(0,T)
\ll_\eps \exp( - c_0 (\log x)^{1/3} (\log\log x)^{-1/3})
$$
holds with any fixed $c_0 < \eps a$.
We also use~\cite[Equation~(18)]{Gal} to
bound the second double sum in~\cite[Equation~(16)]{Gal}. 
Putting everything together, we have
\begin{equation}
\label{eq:gameofthrones}
\begin{split}
&\psi(x+h;q,a)-\psi(x;q,a)-\frac{h}{\varphi(q)}\\
&\qquad\ll_\eps \frac{h}{\varphi(q)}\exp(-c_0(\log x)^{1/3}(\log\log x)^{-1/3})
+\frac{x}{T\varphi(q)}(\log x)^{O(1)}.
\end{split}
\end{equation}
Since
$$
\frac{x}{T}=qx^{1-1/b+\eps/b}\le hx^{\eps/b-\eps}
$$
and $\eps/b-\eps<0$, the first term in the bound of \eqref{eq:gameofthrones}
dominates, and the result follows. 

\section{Comments}

Our results can be extended to more general
classes of moduli. For example, suppose that $q = rs$ with coprime
positive integers $r$ and $s$, and
instead of~\eqref{eq:minmax} we have
$$
\min\limits_{p\mid s}\{v_p(s)\}
\ge 0.7\gamma\qquad
\text{with}\quad\gamma=\max\limits_{p\mid s}\{v_p(s)\}\ge\gamma_0.
$$
For any primitive character $\chi$ modulo $q$, we write
$$
S_\chi(M,N)=\sum_{k=0}^{r-1}
\sum_{(M-k)/r<m\le (M+N-k)/r}\chi(k+rm)+O(r).
$$
Defining $\chi^*(m)=\chi(k+rm)$, we see that
$\chi^*$ is a primitive character modulo $s$,
hence Theorem~\ref{thm:mainG} applies to the inner sum over $m$.
Consequently, if $r$ is not too large (say, $r = N^{o(1)}$),
then we obtain a result of roughly the same 
strength as Theorem~\ref{thm:mainG}. 
This applies to the other results of this paper as well.

\end{document}